\documentclass[reqno]{amsart}

\usepackage[utf8]{inputenc}
\usepackage{amsmath, amssymb}

\usepackage{xcolor}
\usepackage[colorlinks,
    linkcolor={red!50!black},
    citecolor={blue!50!black},
    urlcolor={blue!80!black}]{hyperref}

\usepackage{enumerate}
\usepackage{tikz}
\usetikzlibrary{calc, matrix, arrows, cd, decorations.markings, knots}
\usetikzlibrary{decorations.pathmorphing}
\usetikzlibrary{decorations.pathreplacing}

\usepackage[all]{xypic}

\theoremstyle{plain}
\newtheorem{proposition}{Proposition}[section]
\newtheorem{theorem}[proposition]{Theorem}
\newtheorem{corollary}[proposition]{Corollary}
\newtheorem{lemma}[proposition]{Lemma}

\theoremstyle{definition}
\newtheorem{definition}[proposition]{Definition}

\newtheorem{conjecture}[proposition]{Conjecture}

\theoremstyle{remark}
\newtheorem{remark}[proposition]{Remark}

\newcommand{\wt}{\widetilde}

\newcommand{\ol}{\overline}
\newcommand{\sm}{\setminus}

\newcommand{\Q}{\mathbb{Q}}
\newcommand{\Z}{\mathbb{Z}}

\newcommand{\Id}{\operatorname{Id}}

\newcommand{\Ext}{\operatorname{Ext}}
\newcommand{\Tor}{\operatorname{Tor}}

\newcommand{\Hom}{\operatorname{Hom}}

\newcommand{\ord}{\operatorname{ord}}
\renewcommand{\top}{\operatorname{top}}

\newcommand{\zzm}{\Z[\Z^m]}

\renewcommand{\epsilon}{\varepsilon}
\renewcommand{\phi}{\varphi}

\newcommand{\opnormal}[1]{\operatorname{\textnormal{#1}}}

\begin{document}

\title[Homotopy ribbon concordance]{Homotopy ribbon concordance and Alexander polynomials}

\author{Stefan Friedl}
\address{
  Department of Mathematics\\
  Universit\"{a}t Regensburg\\
  Germany
  }
\email{sfriedl@gmail.com}

\author{Mark Powell}
\address{
  Department of Mathematical Sciences\\
  Durham University\\
  UK
}
\email{mark.a.powell@durham.ac.uk}


\def\subjclassname{\textup{2010} Mathematics Subject Classification}
\expandafter\let\csname subjclassname@1991\endcsname=\subjclassname
\expandafter\let\csname subjclassname@2000\endcsname=\subjclassname

\begin{abstract}
We show that if a link $J$ in the 3-sphere is homotopy ribbon concordant to a link $L$ then the Alexander polynomial of $L$ divides the Alexander polynomial of $J$. 
\end{abstract}

\maketitle

\section{Introduction}
Let $I:= [0,1]$.  An oriented, ordered $m$-component link $J$ in $S^3$ is \emph{homotopy ribbon concordant} to an oriented, ordered $m$-component link $L$ if there is a concordance $C \cong \bigsqcup^m S^1 \times I$, locally flatly embedded in $S^3 \times I$, restricting to $J \subset S^3 \times \{0\}$ and $-L \subset S^3 \times \{1\}$, such that the induced map on fundamental groups of exteriors
\[\pi_1(S^3\sm \nu J) \twoheadrightarrow \pi_1((S^3 \times I) \sm \nu C)\]
is surjective and the induced map
\[\pi_1(S^3\sm \nu L) \rightarrowtail \pi_1((S^3 \times I) \sm \nu C)\]
is injective. Here $\nu J$, $\nu L$, and  $\nu C$ denote open tubular neighbourhoods.   When $J$ is homotopy ribbon concordant to $L$  we write $J \geq_{\top} L$.  From now on we write
\[X_J := S^3\sm \nu J,\;\; X_L := S^3\sm \nu L, \text{ and } X_C := (S^3 \times I) \sm \nu C.\]

The notion of homotopy ribbon concordance is a natural homotopy group analogue of the notion of smooth ribbon concordance initially introduced by Gordon~\cite{Gordon-ribbon-conc} for knots:  we say the link $J$ is smoothly ribbon concordant to  the link $L$, written $J \geq_{\operatorname{sm}} L$, if there is a smooth concordance from $J$ to $L$ such that the restriction of the projection map $S^3 \times I \to I$ to $C$ yields a Morse function on $C$ without minima.  The exterior of such a concordance admits a handle decomposition relative to $X_J$ with only 2- and 3-handles, from which it is easy to see that the induced map $\pi_1(X_J) \to \pi_1(X_C)$ is surjective.  Gordon's argument~\cite[Lemma~3.1]{Gordon-ribbon-conc} shows that $\pi_1(X_L) \to \pi_1(X_C)$ is injective. Thus a smooth ribbon concordance is a homotopy ribbon concordance.

We define the Alexander polynomial $\Delta_J(t_1,\dots,t_m) \in \Z[t_1^{\pm 1},\dots,t_m^{\pm 1}]$ of an oriented, ordered $m$-component link $J$ to be the order of the  \emph{torsion} submodule of the Alexander module $H_1(X_J;\Z[\Z^m])$. Here the precise coefficient system $\phi \colon \pi_1(X_J)\to \Z^m$ is determined by the oriented meridians and the ordering of~$L$.

\begin{theorem}\label{theorem:alexander-polynomials-divide}
  Suppose that $J \geq_{\operatorname{top}} L$. Then $\Delta_L \mid \Delta_J$.
\end{theorem}

For knots and for $\geq_{\operatorname{sm}}$ instead of $\geq_{\operatorname{top}}$, Theorem~\ref{theorem:alexander-polynomials-divide} is a consequence of a more general theorem of Gilmer~\cite{Gilmer-ribbon-conc}. However Gilmer's proof does not extend to the topological category.

Further classical work on smooth ribbon concordance includes~\cite{Miya-ribbon-conc},\cite{Gilmer-ribbon-conc}, \cite{Miya-ribbon-conc-2}, and \cite{Silver-ribbon-conc}.

We want to explain a fairly simple proof of Theorem~\ref{theorem:alexander-polynomials-divide}, thus we will not prove the most general result possible.  But we expect that our argument can be generalised to twisted Alexander polynomials~\cite{Kirk-Livingston:1999-2, Kirk-Livingston:1999-3, HKL08} and higher order Alexander polynomials~\cite{Cochran:2002-1}, provided one uses a unitary representation that extends over the ribbon concordance exterior.
Our proof can also be generalised to concordances between links in homology spheres.  Having not found a convincing application, we have not carried out either of these generalisations in this short note.


A number of articles have recently appeared on the relation of smooth ribbon concordance to Heegaard-Floer and Khovanov homology~\cite{Zemke1, Zemke-Levine, Zemke-Miller, Juhasz-Miller-Zemke, Sarkar}.  These techniques of course do not apply to locally flat concordances.  We thought it might be of interest to show how to establish, in many cases and with minimal machinery, that two concordant links are not ribbon concordant, in both categories. 

\begin{remark}
  It is straightforward to apply Theorem~\ref{theorem:alexander-polynomials-divide} to construct examples of concordant knots that are not homotopy ribbon concordant.  For instance (this example was given by Gordon~\cite{Gordon-ribbon-conc}, but with a different proof), let $K$ be a trefoil and let $J$ be the figure eight knot. Then $K\# -K$ and $J \# -J$ are both slice, so are concordant. But the Alexander polynomials are coprime, so there is no homotopy ribbon concordance between these knots.
\end{remark}

\begin{remark} Perhaps somewhat surprisingly, the condition that $\pi_1(X_L) \to \pi_1(X_C)$ is injective is not needed anywhere in our proof of Theorem~\ref{theorem:alexander-polynomials-divide}.

Gordon conjectured that smooth ribbon concordance gives a partial order on knots.  This conjecture is still open: in order to prove it, one would have to show that if $J$ is smoothly ribbon concordant to $K$ and $K$ is smoothly ribbon concordant to $J$, then $K$ and $J$ are isotopic.

In the topological category, by work of Freedman~\cite[Theorem~11.7B]{Freedman-Quinn:1990-1}, there is a concordance $C$ with $\pi_1(X_C)\cong \Z$ from the unknot $U$ to $K$ for every Alexander polynomial one knot $K$.
So in order to make the analogous conjecture that $\geq_{\top}$ is a partial order, one certainly needs that $\pi_1(X_K) \to \pi_1(X_C)$ is injective, and we have included it in the definition. Thus, the concordance $C$ is not a homotopy ribbon concordance. 
\end{remark}

We conclude this introduction with the following conjecture that is the topological analogue of Gordon's Conjecture.

\begin{conjecture}
  The relation $\geq_{\top}$ is a partial order on the set of knots.
\end{conjecture}

\subsection*{Acknowledgements}
We would like to thank Arunima Ray and the Max Planck Institute for Mathematics in Bonn. We also thank our first anonymous referee for providing the impetus to include the case of links and we would like to thank our second referee for a thoughtful referee report.
SF was supported by the SFB 1085 ``higher invariants'' which is supported by the Deutsche Forschungsgemeinschaft DFG.

\section{Twisted homology and cohomology}
As preparation for the proofs in the following section we recall the definitions of twisted (co-) homology modules.

Given a group $\pi$ and a left $\Z\pi$-module $A$, we
write $\ol{A}$ for the right $\Z\pi$-module that has the same underlying abelian group but for which the right action of $\Z\pi$ is defined by $a\cdot g:=g^{-1}\cdot a$ for $a\in A$ and $g\in \pi$.  The same notation is also used with the r\^{o}les of left and right reversed and $g \cdot a := a \cdot g^{-1}$.
Here is the definition of twisted homology and cohomology groups.

\begin{definition}
Let $X$ be a connected topological space that admits a universal cover $p\colon \wt{X}\to X$.
Write $\pi:=\pi_1(X)$.
Let $Y$ be a subset of $X$ and let $A$ be a right $\Z\pi$-module.
Let $\pi$ act on $\wt{X}$ by deck transformations, which is naturally a left action. Thus, the singular chain complex  $C_*(\wt{X},p^{-1}(Y))$ becomes a left $\Z\pi$-module chain complex.
Define the \emph{twisted chain complex}
\[ C_*(X,Y;A)\,\,:=\,\,\big( A\otimes_{\Z\pi} C_*(\wt X, p^{-1}(Y)),\Id\otimes \partial_*\big).\]
The corresponding \emph{twisted homology groups} are $H_k(X,Y;A)$.
With~$\delta^k = \Hom(\partial_k,\Id)$ define the \emph{twisted cochain complex} to be
\[ C^*(X,Y;A)\,\,:=\,\,\big(\Hom_{\opnormal{right-}\Z\pi}\big(\ol{C_*(\wt{X},p^{-1}(Y)}), A\big), \delta^*\big).\]
The corresponding \emph{twisted cohomology groups} are $H^k(X,Y;A)$.
\end{definition}

If $R$ is some ring and $A$ is an $(R,\Z\pi)$-bimodule, then the above twisted homology and cohomology groups are naturally left $R$-modules.

In this article, $X$ will be one of $X_J$, $X_L$, or $X_C$, and we will have $A= \Z[\Z^m]$, considered as a $(\Z[\Z^m],\Z\pi)$-bimodule, with the left action by left multiplication and with the right $\Z\pi$ action induced by the homomorphism
\[\pi=\pi_1(X) \to H_1(X;\Z) \xrightarrow{\cong} \Z^m.\]
Here the first map is the Hurewicz map and the isomorphism is determined by the orientations and the ordering of the link components.
We refer to the $\Z[\Z^m]$-modules $H_1(X_B;\Z[\Z^m])$, for $B \in \{J,L,C\}$, as the \emph{Alexander module} of $J$, $L$, and $C$ respectively.
We shall also make use of the analogous twisted homology and cohomology modules of the pairs $(X_C,X_J)$ and $(X_C,X_L)$.

\section{Injection and surjection of Alexander modules}
In this section we will prove several results on the interplay between Alexander modules and homotopy ribbon concordance. The combination of these results will imply
Theorem~\ref{theorem:alexander-polynomials-divide}.

\begin{proposition}\label{prop:surjection}
If $C$ is a homotopy ribbon concordance from $J$ to $L$, then the induced map
  \[H_1(X_J;\Z[\Z^m]) \to H_1(X_C;\Z[\Z^m]) \]
  is surjective.
 \end{proposition}

\begin{proof}[First proof of Proposition~\ref{prop:surjection}]
  Consider the following commutative diagram
\[ \xymatrix@R0.65cm{ 1\ar[r] & K_J:=\ker(\pi_1(X_J)\to \Z^m)\ar[d]\ar[r]& \pi_1(X_J)\ar[r]\ar@{->>}[d] & \Z^m\ar[r]\ar[d]^=&0\\
 1\ar[r] & K_C:=\ker(\pi_1(X_C)\to \Z^m)\ar[r]& \pi_1(X_C)\ar[r] & \Z^m\ar[r]&0.}\]
Since the middle map is an epimorphism we see that map on the left is an epimorphism.  For any group epimorphism $G \to H$, the induced map on abelianisations $G^{ab} \to H^{ab}$ is an epimorphism, so in particular the induced map $K_J^{ab} \to K_C^{ab}$ is an epimorphism.  Note that $K_J$ and $K_C$ are the fundamental groups of the universal abelian covering spaces $\ol{X}_J$ and $\ol{X}_C$ of $X_J$ and $X_C$ respectively.  The Hurewicz theorem identifies the abelianisation of the fundamental group of a path connected space with the first homology, so that
\[\xymatrix@R0.65cm{K_J^{ab} \ar[r] \ar[d]^{\cong} & K_C^{ab} \ar[d]^{\cong} \\
 H_1(\ol{X}_J;\Z) \ar[r] & H_1(\ol{X}_C;\Z) }\] commutes. It follows that the map on the bottom row is an epimorphism.
But by the topologists' Shapiro lemma~\cite[p.~100]{Davis-Kirk} the homology groups $H_1(\ol{X}_J;\Z)$ and $H_1(\ol{X}_C;\Z)$ are naturally isomorphic to the twisted homology groups $H_1(X_J;\Z[\Z^m])$ and $H_1(X_C;\Z[\Z^m])$ respectively.
\end{proof}

Here is another proof using homological algebra, for which generalisation to twisted coefficients would be easier.

\begin{proof}[Second proof of Proposition~\ref{prop:surjection}]
We prove the somewhat stronger statement that $H_1(X_C,X_J;\zzm) =0$.
Consider the long exact sequence of the pair with $\Z\pi:= \Z[\pi_1(X_C)]$ coefficients, where $\pi:= \pi_1(X_C)$:
\begin{align*}
  H_1(X_C;\Z\pi) &\to H_1(X_C,X_J;\Z\pi) \to H_0(X_J;\Z\pi) \\  \to  H_0(X_C;\Z\pi) &\to H_0(X_C,X_J;\Z\pi) \to 0
\end{align*}
Since $\pi= \pi_1(X_C)$, we have $H_1(X_C;\Z\pi)=0$ and $H_0(X_C;\Z\pi) \cong \Z$.
Since $\pi_1(X_J) \to \pi$ is surjective, the pull-back cover
\[\xymatrix@R0.65cm{\ol{X}_J \ar[r] \ar[d] & \wt{X}_C \ar[d] \\
 X_J \ar[r] & X_C, }\]
where $\wt{X}_C \to X_C$ is the universal cover, is precisely the connected cover of $X_J$ corresponding to the subgroup $\ker(\pi_1(X_J) \to \pi_1(X_C))$.
It follows that $H_0(X_J;\Z\pi) \cong \Z$ and the map $H_0(X_J;\Z\pi) \to H_0(X_C;\Z\pi)$ is an isomorphism.  We deduce that
\[H_1(X_C,X_J;\Z\pi)=0 = H_0(X_C,X_J;\Z\pi).\]
Next, apply the universal coefficient spectral sequence for homology (see \cite[Theorem~10.90]{Rotman:2009})
\[\Tor_p^{\zzm}(H_q(X_C,X_J;\Z\pi),\zzm) \, \Rightarrow \, H_{p+q}(X_C,X_J;\zzm).\]
 to change to $\zzm$ coefficients. The terms on the 1-line ($p+q=1$) of the  $E^2$ page are
\[\zzm \otimes_{\Z\pi} H_1(X_C,X_J;\Z\pi) =0\text{ and } \Tor_1^{\Z\pi}(H_0(X_C,X_J;\Z\pi),\zzm) =0.\]
It follows that the 1-line on the $E^{\infty}$ page vanishes too, so that $H_1(X_C,X_J;\zzm)=0$ as desired. This completes the proof of
the proposition.
\end{proof}

We continue with the following variation on Proposition~\ref{prop:surjection}.

\begin{proposition}\label{prop:surjection-torsion}
If $C$ is a homotopy ribbon concordance from $J$  to $L$, then the induced map
  \[TH_1(X_J;\Z[\Z^m]) \to TH_1(X_C;\Z[\Z^m]) \]
between the $\Z[\Z^m]$-torsion submodules is surjective.
 \end{proposition}

\begin{proof}
  First, the fact that $X_J \to X_C$ induces a $\Z$-homology isomorphism implies that $H_i(X_C,X_J;\Z)=0$ for all $i$. By chain homotopy lifting~\cite[Proposition~2.10]{Cochran-Orr-Teichner:1999-1} this implies that \[H_i(X_C,X_J;\Q(\Z^m)) =0\] for all $i$.
This in turn implies that the right vertical map in the next commutative diagram is an isomorphism.
\[ \xymatrix@R0.65cm{ 0\ar[r] &TH_1(X_J;\Z[\Z^m])\ar[r] \ar[d] & H_1(X_J;\Z[\Z^m])\ar@{->>}[d]\ar[r] & H_1(X_J;\Q(\Z^m))\ar[d]^\cong \\
 0\ar[r] &TH_1(X_C;\Z[\Z^m])\ar[r] & H_1(X_C;\Z[\Z^m])\ar[r] & H_1(X_C;\Q(\Z^m)) }\]
Since $\Q(\Z^m)$ is flat over $\Z[\Z^m]$, the horizontal sequences are exact. By Proposition~\ref{prop:surjection} the middle map is an epimorphism.
A straightforward diagram chase shows that the left vertical map is also an epimorphism.
\end{proof}

The following corollary is an immediate consequence of
Proposition~\ref{prop:surjection-torsion} and of the multiplicativity of orders in short exact sequences of torsion $\zzm$-modules~\cite[Lemma~5]{Levine67}.

\begin{corollary}\label{cor:alex-polys-divide-surjective}
The orders of the torsion submodules of the homologies satisfy:
\[\ord TH_1(X_C;\Z[\Z^m]) \mid \underset{=\Delta_J}{\underbrace{\ord TH_1(X_J;\Z[\Z^m])}}.\]
  \end{corollary}

We continue with the following proposition that relates the Alexander modules of $J$ and $C$.

\begin{proposition}\label{prop:alex-polys-divide-injective}
If $C$ is a homotopy ribbon concordance from $J$ to $L$, then the induced map
  \[H_1(X_L;\Z[\Z^m]) \to H_1(X_C;\Z[\Z^m]) \]
  is injective.
\end{proposition}

In the proof of  Proposition~\ref{prop:alex-polys-divide-injective}  we shall make use of the next lemma. The proof of the lemma is a straightforward check and is omitted.

\begin{lemma}\label{lem:switch-to-lambda-chain-complex}
Let $\pi$ be a group, let $C_*$ be a chain complex of free left $\Z[\pi]$-modules and let $\phi\colon \pi\to \Z^m$ be a homomorphism. The map $\phi$ induces a $(\Z[\Z^m],\Z[\pi])$-bimodule structure on $\Z[\Z^m]$.
The map
\[ \begin{array}{rcl}  \Hom_{\operatorname{right-}\Z[\pi]}(\ol{C_*};\Z[\Z^m])&\to &
  \ol{\Hom_{\Z[\Z^m]}(\Z[\Z^m]\otimes_{\Z[\pi]} C_*;\Z[\Z^m])}\\
f&\mapsto & (p\otimes \sigma\mapsto p\cdot \ol{f(\sigma)})
\end{array}
\]
is well-defined and is an isomorphism of $\Z[\Z^m]$-cochain complexes.
\end{lemma}

\begin{proof}[Proof of Proposition~\ref{prop:alex-polys-divide-injective}]
We show that $H_2(X_C,X_L;\zzm)=0$.
As in the proof of Proposition~\ref{prop:surjection-torsion},
$H_i(X_C,X_L;\Q(\Z^m)) =0$ for all $i$. Since commutative localisation is flat, this implies in particular that $H_i(X_C,X_L;\zzm)$ is $\zzm$-torsion for all $i$.

Now by Poincar\'{e}-Lefschetz duality (see e.g.~\cite[Theorem~A.15]{FNOP-guide} for a proof with twisted coefficients in the topological category),
\[H_2(X_C,X_L;\zzm) \cong H^2(X_C,X_J;\zzm).\]

As above write $\pi:= \pi_1(X_C)$.
Now \[H^2(X_C,X_J;\zzm) \cong \overline{H^2(\Hom_{\zzm}(\zzm \otimes_{\Z\pi} C_*(X_C,X_L;\Z\pi), \zzm))}\]
by Lemma~\ref{lem:switch-to-lambda-chain-complex}.
We can compute the right hand side of this using the universal coefficient spectral sequence for cohomology \cite[Theorem~2.3]{Levine-77-knot-modules}, which combined with the equation above gives a spectral sequence 
\[\Ext^p_{\zzm}(H_q(X_C,X_J;\zzm),\zzm) \, \Rightarrow \, \overline{H^{p+q}(X_C,X_J;\zzm)}.\]
We shall show that all the terms on the 2-line ($p+q=2$) vanish.
First, since $H_2(X_C,X_J;\zzm)$ is torsion, we have
\[\Ext^0_{\zzm}(H_2(X_C,X_J;\zzm),\zzm) \cong \Hom_{\zzm}(H_2(X_C,X_J;\zzm),\zzm) =0.\]
We showed in the proof of Proposition~\ref{prop:surjection} that $H_1(X_C,X_J;\zzm)=0$. Therefore \[\Ext^1_{\zzm}(H_1(X_C,X_J;\zzm),\zzm) =0.\] Finally  $H_0(X_C,X_J;\zzm) =0$, so \[\Ext^2_{\zzm}(H_0(X_C,X_J;\zzm),\zzm) =0.\]  This completes the proof that all the terms on the 2-line vanish, so we see that
\[\overline{H_2(X_C,X_L;\zzm)} \cong \overline{H^2(X_C,X_J;\zzm)} =0\]
which implies that $H_2(X_C,X_L;\zzm) =0$ as desired.
It then follows from the long exact sequence of the pair $(X_C,X_L)$ that the map  \[H_1(X_L;\Z[\Z^m]) \to H_1(X_C;\Z[\Z^m]) \]
  is injective.
  \end{proof}

Using the aforementioned  multiplicativity of orders in short exact sequences of torsion $\zzm$-modules we immediately obtain the following corollary.

\begin{corollary}\label{cor:alex-polys-divide-injective}
The orders of the torsion submodules of the homologies satisfy:
\[\underset{=\Delta_L} {\underbrace{\ord TH_1(X_L;\Z[\Z^m])}} \mid \ord TH_1(X_C;\Z[\Z^m]).\]
 \end{corollary}


\section{Proof of Theorem~\ref{theorem:alexander-polynomials-divide}}
    By Corollary~\ref{cor:alex-polys-divide-injective}, we have that  $\Delta_L = \ord  TH_1(X_L;\zzm)$ divides $\Delta_C := \ord TH_1(X_C;\zzm)$. That is, $\Delta_C = \Delta_L \cdot p$ for some $p \in \zzm$.  On the other hand, by Corollary~\ref{cor:alex-polys-divide-surjective}, for some $q \in \zzm$ we have that $\Delta_C \cdot q = \Delta_J$. Therefore \[\Delta_J = \Delta_C \cdot q = \Delta_L \cdot p \cdot q\] and so $\Delta_L \mid \Delta_J$ as desired. This completes the proof of Theorem~\ref{theorem:alexander-polynomials-divide}.

\bibliographystyle{alpha}
\def\MR#1{}
\bibliography{research}

\end{document}